\documentclass[11pt]{article}

\usepackage{graphicx}
\usepackage{amsmath}
\usepackage{amsthm}
\usepackage{amssymb}
\usepackage[mathscr]{eucal}

\title{Teichm\"{u}ller geometry of moduli space, II:\\ 
${\cal M}(S)$ seen from far away}
\author{Benson Farb and Howard Masur \thanks{Both authors are
supported in part by the NSF.}}

\theoremstyle{plain}
\newtheorem{theorem}{Theorem}

\newtheorem{lemma}[theorem]{Lemma}
\newtheorem{corollary}[theorem]{Corollary}

\def\proof{{\bf {\medskip}{\noindent}Proof. }}

\def\endproof{$\diamond$ \bigskip}

\def\title{\em}

\newcommand\R{\mbox{\bf R}}

\newcommand\hyp{\mbox{\bf H}}

\newcommand\Q{\mbox{\bf Q}}
\newcommand\Qrank{\Q\mbox{-rank}}

\newcommand\newpar{\medskip \noindent}

\DeclareMathOperator\teich{Teich}
\DeclareMathOperator\Teich{\teich}

\DeclareMathOperator\Mod{Mod}

\DeclareMathOperator\M{{\cal M}}

\DeclareMathOperator\genus{genus}

\DeclareMathOperator\CC{{\cal C}}
\DeclareMathOperator\V{{\cal V}}

\DeclareMathOperator\cone{Cone}

\begin{document}
\maketitle


\section{Introduction}

Let $S=S_{g,n}$ be a closed, orientable surface with genus $g\geq 0$ 
with $n\geq 0$ marked points, and let $\Teich(S)$ be the associated
Teichm\"{u}ller space of marked conformal classes or (fixed area) constant curvature 
metrics on $S$.  Endow $\Teich(S)$ with the Teichm\"{u}ller metric $d_{\Teich(S)}(\cdot,\cdot)$.   
Recall that for marked conformal structures $X_1,X_2\in\Teich(S)$ we define
$$d_{\Teich(S)}(X_1,X_2)=\frac{1}{2}\log K$$
where $K\geq 1$ is the least number such that there is a $K$-quasiconformal mapping between the marked structures $X_1$ and $X_2$.   The
mapping class group $\Mod(S)$ acts properly discontinuously and
isometrically on $\Teich(S)$, thus inducing a metric $d_{\M(S)}(\cdot,\cdot)$ on the quotient
moduli space $\M(S):=\Teich(S)/\Mod(S)$.  Let $\pi:\Teich(S)\to \M(S)$ be the natural projection. 

The goal of this paper is to build an ``almost isometric'' simplicial 
model for $\M(S)$, from which we will determine the tangent
cone at infinity of $\M(S)$.   In analogy with the case of locally symmetric spaces, this can 
be viewed as a step in building a ``reduction theory'' for the action of 
$\Mod(S)$ on $\Teich(S)$.  Other results in this direction can be found in \cite{Le}.

\newpar
{\bf Moduli space seen from far away. }Gromov formalized the idea of
``looking at a metric space $(X,d)$ from far away'' by 
introducing the notion of the
{\em tangent cone at infinity} of $(X,d)$.  This metric space, denoted 
$\cone(X)$, is defined to be a Gromov-Hausdorff limit of based 
metric spaces (where basepoint $x\in X$ is fixed once and for all):

$$\cone(X):=\lim_{\epsilon\to 0}\ (X,\epsilon d)$$

So, for example, any compact Riemannian manifold $M$ has
$\cone(X)=\ast$, a one point space.  
Let $M=\Gamma\backslash G/K$ be an arithmetic, locally symmetric 
manifold (or orbifold); so $G$ is a semisimple algebraic $\Q$-group, $K$ a maximal compact
subgroup, and $\Gamma$ an arithmetic lattice.  
Hattori, Leuzinger and Ji-MacPherson proved that $\cone(M)$ is a metric
cone over the quotient by $\Gamma$ 
of the spherical Tits building $\Delta_{\Q}(G)$
associated to $G_{\Q}$.  Here the metric on the cone on a maximal simplex of 
$\Delta_{\Q}(G)$ makes it isometric to the standard (Euclidean) metric 
on a Weyl chamber in $G/K$.  In particular they deduce:
$$\Qrank(\Gamma)=\dim(\cone(\Gamma\backslash G/K))$$

Our first result is a determination of the metric space 
$\cone(\M(S))$.  The role of the
rational Tits building will be played by the {\em complex of curves} $\CC(S)$ 
on $S$.  Recall that, except for some sporadic cases discussed below,
the complex $\CC(S)$ is defined to be the simplicial complex whose
vertices are (isotopy classes of) simple closed curves on $S$, and whose
$k$-simplices are $(k+1)$-tuples of distinct isotopy classes which can
be realized simultaneously as disjoint curves on $S$.  Note that
$\CC(S)$ is a $d$-dimensional simplicial complex, where $d=3g-4+n$.  While $\CC(S)$ is
locally infinite, its quotient by the natural action of $\Mod(S)$ is a 
{\em finite orbicomplex}, by which we mean a finite simplicial
complex where each simplex is quotiented out by the action of a finite group.   The quotient can be made a simplicial complex by looking at the action on the barycentric subdivision of $\CC(S)$.
Denote by $P$ the natural quotient map $$P:\CC(S)\to \CC(S)/\Mod(S).$$
We now build a metric space which will serve as a coarse metric model
for $\M(S)$.  Let $\widetilde{\V}(S)$ denote the topological cone 
$$\widetilde{\V}(S):=
\frac{\displaystyle [0,\infty)\times \CC(S)}{\displaystyle 
\{0\}\times \CC(S)}$$ 
For each maximal simplex $\sigma$ of $\CC(S)$, we will think of the cone over 
$\sigma$ as an orthant with coordinates $(x_1,\ldots ,x_d)$.  We endow this orthant with the 
standard $\sup$ metric:
$$d((x_1,\ldots ,x_{d}), (y_1,\ldots,y_d)):=\frac{1}{2}
\max_{1 \leq i \leq d}|x_i-y_i|.$$
The factor of $\frac{1}{2}$ is designed to be consistent with the definition of the Teichm\"{u}ller metric.

The metrics on the cones on any two such maximal simplices clearly agree
on (the cone on) any common face.  We can thus endow $\widetilde{\V}(S)$
with the corresponding path metric.  Note that the natural action of
$\Mod(S)$ on $\widetilde{\V}(S)$ induces an isometric action of $\Mod(S)$ on
$\widetilde{\V}(S)$.  The quotient
$$\V(S):=\widetilde{\V}(S)/\Mod(S)$$ 
thus inherits a well-defined metric. 
The example $\V(S_{1,2})$ is 
described in Figure \ref{figure:t12}.  To endow $\V(S)$ with 
the structure of a simplicial complex instead of an orbicomplex, we can simply replace 
$\CC(S)$ with its barycentric subdivision in the construction above.

Our main result is that $\V(S)$ provides a simple and reasonably
accurate geometric model for $\M(S)$.

\begin{theorem}
\label{theorem:model}
There is a {\em $(1,D)$-quasi-isometry} $\Psi:\V(S)\to \M(S)$.  
 That is, there is a constant $D=D(S)\geq 0$ such
that :
\begin{itemize}
\item $|d_{\V(S)}(x,y)-d_{\M(S)}(\Psi(x),\Psi(y))|\leq D$ 
for each $x,y\in \V(S)$, and
\item The $D$-neighborhood of $\Psi(\V(S))$ in
$\M(S)$ is all of $\M(S)$.
\end{itemize}
\end{theorem}

The main ingredient in our proof of Theorem \ref{theorem:model} is a
theorem of Minsky \cite{Mi}, which determines up to an additive factor
the Teichm\"{u}ller metric near infinity in $\Teich(S)$.

\begin{figure}
\label{figure:t12}
\begin{center}
\includegraphics{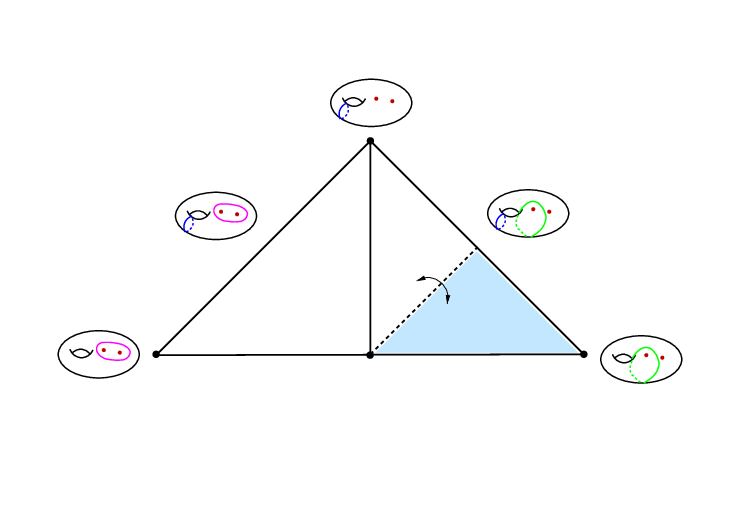}
\end{center}
\vspace{2in}
\caption{The metric space $\V(S_{1,2})$.  The fundamental domain for the
action of $\Mod(S)$ on $\CC(S)$ is the union of two edges, one
corresponding to a separating/nonseparating pair of curves, the other to
a nonseparating/nonseparating pair. These are the only combinatorial types. 
 Note that the latter edge has an order two symmetry, corresponding to the mapping 
class which switches
the curves.  Thus $\V(S)$ is the union of a Euclidean
quadrant and a quotient of a Euclidean quadrant by a reflection along
the $y=x$ ray.}
\end{figure}

It is clear that Theorem \ref{theorem:model} implies that 
$\cone(\M(S))=\cone(\V(S))$.   Further, it is clear that multiplying the given metric on 
$\V(S)$ by any fixed constant gives a metric space which is isometric (via the dilatation) 
to the original metric.  In particular, $\cone(\V(S))$ is isometric to $\V(S)$ itself.   
We thus deduce the following.

\begin{corollary}
\label{corollary:faraway}
$\cone(\M(S))$ is isometric to $\V(S)$.
\end{corollary}

Using different methods, 
Leuzinger \cite{Le} has independently proven that $\V(S)$ is 
bilipschitz homeomorphic to $\cone(\M(S))$.  His methods do not seem to yield the isometry type of 
$\cone(\M(S))$. 

\begin{remarks}
\begin{enumerate}

\item Corollary \ref{corollary:faraway} has applications to metrics of
positive scalar curvature.  Namely, it is a key ingredient in the proof 
by Farb-Weinberger that, while $\M(S)$ admits a metric
of positive scalar curvature for most $S$ (e.g. when $\genus(S)>2$), it
admits no metric with the same quasi-isometry type as the
Teichm\"{u}ller metric on $\M(S)$.  See \cite{FW}.

\item For locally symmetric $M$, we know that $\cone(M)$ is
nonpositively curved in the ${\rm CAT}(0)$ sense.  In contrast, 
$\V(S)$ strongly exhibits aspects of {\em positive} curvature, since
even within the cone on a single simplex, any
two points $x,y\in \V(S)$ have whole families of distinct geodesics 
between them, and these geodesics get arbitrarily far apart as
$d(x,y)\to\infty$.  This is a basic property of the $\sup$ metric on a quadrant.

\item Corollary \ref{corollary:faraway} implies that any metric on $\M(S)$ 
quasi-isometric to the Teichm\"{u}ller metric must have a cone which is bilipshitz homeomorphic 
to $\V(S)$.   
\end{enumerate}
\end{remarks}

The authors would like to thank the referee for some extremely helpful comments. 

\section{The proof of Theorem \ref{theorem:model}}

\subsection{Minsky's Product Theorem}

In this subsection we recall some work of Minsky which will be crucial for what follows.  

Let $d=3g-3+n$.  Fix $\epsilon>0$ smaller than the Margulis
constant for hyperbolic surfaces.  
Let ${\cal C}=\{\gamma_1,\ldots,\gamma_p\}$ be a collection of distinct, 
disjoint, nontrivial homotopy classes of simple closed curves; this is a simplex in $\CC(S)$.   Let 
$$\Omega_{\cal C}(\epsilon):=\{X\in\Teich(S): \ell_X(\gamma_i)<\epsilon \mbox{\ for each\ }i=1,\ldots,p\}.$$  

Extend ${\cal C}$ to a maximal collection $\{\gamma_1,\ldots,
\gamma_d\}$ of homotopy classes of simple closed curves.  
Let $\{\theta_i,\ell_i)\}$ denote the corresponding Fenchel-Nielsen 
coordinates on $\Omega_{\cal C}(\epsilon)$.  Recall that Fenchel-Nielsen 
coordinates give global coordinates on $\Teich(S)$; henceforth we will 
identify points in $\Teich(S)$ with their corresponding coordinates.  

Consider the Teichm\"{u}ller 
space $\Teich(S\setminus {\cal C})$, which is the space of complete, finite 
area hyperbolic metrics on $S\setminus {\cal C}$.  Note that the coordinates 
$\{(\theta_i,\ell_i): i>p\}$ give Fenchel-Nielsen 
coordinates on $\Teich(S\setminus {\cal C})$.

Let 
$$\Phi=(\Phi_1,\Phi_2):\Omega_{\cal C}(\epsilon)\to \Teich(S\setminus {\cal C})\times\prod_{i=1}^p\hyp^2$$
be defined by
$$\Phi((\theta_1,\ldots,\theta_d, \ell_1,\ldots ,\ell_d,)):=
(\theta_{p+1},\ldots,\theta_d, \ell_{p+1},\ldots ,\ell_d,)\times \prod_{i=1}^p(\theta_i,1/\ell_i).$$
Notice that we are changing the last set of length coordinates from $\ell$ to $1/\ell$ giving coordinates in the upper half-space model of $\hyp^2$. We give $\hyp^2$ the metric $ds^2=\frac{1}{4}(dx^2+dy^2)/y^2$.  Note that the factor of $\frac{1}{4}$ leads to a factor of 
$\frac{1}{2}$ in the distance, and is consistent with the factor of $\frac{1}{2}$ in the metric on the 
Euclidean octant.  If 
 $S\setminus {\cal C}$ is disconnected, then $\Teich(S\setminus {\cal C})$ is itself a 
product of the Teichmuller spaces of the components of $S\setminus {\cal C}$; 
we endow this total  product space itself with the $\sup$ metric,
denoted by $d$.  We remark that 
$\Phi$ is a homeomorphism onto its image, and its image is 
$\Teich(S\setminus{\cal C})\times\prod_{i=1}^p\{(x_i,y_i)\in\hyp^2: y_i>1/\epsilon\}$.

The following was proved in \cite{Mi}.

\begin{theorem}[Minsky Product Theorem]
\label{thm:Min}
With notation as above, there exists $D$ such that for all $X,Y\in 
\Omega_{\cal C}(\epsilon)$,  
$$|d(\Phi(X),\Phi(Y))-d_{\Teich(S)}(X,Y)|\leq D.$$
\end{theorem}

We will need the following lemma about distances in $\M(S)$.

\begin{lemma}
\label{lemma:ratios}
Given constants $C,C'$ there is a constant $C''$ with the following property. Let $\sigma=\{\alpha_1,\ldots ,\alpha_d\}$ be a maximal simplex of $\widetilde{V}(S)$.  
Let $X,Y\in\Teich(S)$ be such that $\ell_X(\alpha_i)\leq C$ and $\ell_Y(\alpha_i)\leq C$ for each $i$.  Suppose also that  $|\log(\ell_X(\alpha_i)/\ell_Y(\alpha_i))|\leq C'$.  Then $d_{\M(S)}(\pi(X),\pi(Y))\leq C''$.
\end{lemma}

\proof
This  follows from Theorem~\ref{thm:Min}.  We can find a
point $Y'$ which differs from $Y$ by Dehn twists about curves in
$\sigma$ so that the Fenchel-Nielsen twist coordinates of $X,Y'$
have bounded difference. Now we consider the list of curves
shorter than $\epsilon$ on both $X$ and $Y'$.  Since the ratios of lengths of these short curves are bounded above, as are the differences in twist coordinates,  it follows that the distances in the corresponding $\hyp^2$ factors are bounded.  The complement of
these short curves determines a boundary Teichmuller space. The lengths
of the remaining curves are bounded above and below, giving that the surfaces have a bounded distance from each other in this boundary Teichm\"{u}ller space.  The existence of $C''$ now follows from Theorem~\ref{thm:Min}.
\endproof

\subsection{Defining the map $\Psi$}

We will define a map $\widetilde{\Psi}:\widetilde{\V}(S)\to\M(S)$
by giving its value on a representative of each $\Mod(S)$-orbit
in $\widetilde{\V}(S)$, and then define $\widetilde{\Psi}$ to be
constant on orbits.  It will then follow that $\widetilde{\Psi}$
induces a map $\Psi:\V(S)\to\M(S)$.  While this map will not be
continuous, we will prove that it is a $(1,D)$-quasi-isometry for
some $D\geq 0$. 

Fix a (finite) collection of maximal simplices that represent all combinatorial types.
We will first define $\widetilde{\Psi}$ on the open cone over this collection.   
Thus let $\sigma$ be one of these maximal simplices of
$\CC(S)$ representing  a maximal collection of
disjoint simple closed 
curves $\{\alpha_1,\ldots ,\alpha_{d}\}$.  Again we think of the cone on $\sigma$,
as a subspace of $\widetilde{\V}(S)$, as an octant in
$\R^d$ with coordinates $x_1,\ldots ,x_d$, endowed with the
$\sup$ metric. Let $\Mod(S,\sigma)$  be the subgroup of $\Mod(S)$ that fixes $\sigma$. It acts on the open cone
over $\sigma$ with finite orbit. Take a sector $\Lambda(\sigma)$ inside this cone which is a fundamental domain for the action of $\Mod(S,\sigma)$.   
For any $(x_1,\ldots ,x_d)\in \Lambda(\sigma)$  (no $x_i=0$), let  
\begin{equation}
\label{eq:define}
\widetilde{\Psi}(x_1,\ldots ,x_d):=\pi(X)
\end{equation}
where $\pi(X)$ is any  point of $\pi(\Omega_\sigma(\epsilon))$ 
such that $$\ell_X(\alpha_i)=\epsilon e^{-x_i} \ \ \mbox{for each $i$}.$$
Using the action of $\Mod(S,\sigma)$  we extend $\widetilde{\Psi}$ to the entire open cone on $\sigma$.
Note that $\widetilde\Psi$ is continuous on each open cone. We do this for each maximal cone in the finite collection. 
Now use the action of $\Mod(S)$ to extend $\widetilde{\Psi}$ to the open cones on all maximal simplices by having it be constant on orbits. 
 
  Next  let $\tau$ be a simplex which is not maximal.  Choose some closed maximal simplex $\sigma=\sigma(\tau)$ containing $\tau$.  We call this the maximal simplex {\em associated} to $\tau$. The cone on $\tau$ is given by the coordinates $(x_1,\ldots ,x_d)$ for the cone on $\sigma$ as above.    The coordinates
 $x_i$ corresponding to curves in $\sigma-\tau$ are set to $0$.  Define $\widetilde{\Psi}$ on $\tau$ via the equation (\ref{eq:define}) above.  Thus  all curves in $\sigma-\tau$ are assigned the fixed length $\epsilon$ while the curves in $\tau$ can have arbitrarily small length. We extend $\widetilde{\Psi}$ to all of $\widetilde{\V}(S)$ by declaring $\widetilde{\Psi}$ to be constant on each $\Mod(S)$-orbit in $\widetilde{\V}(S)$.  It follows that $\widetilde{\Psi}$ induces a map $\Psi:\V(S)\to\M(S)$.  We remark that $\Psi$ will in general not be continuous  because of the choices made at a face of a maximal  simplex.
Nonetheless we want to know that the jump in the function at any face is uniformly bounded. 
We will argue this below using 
Lemma \ref{lemma:ratios} together with the following lemma.

\begin{lemma}
\label{lemma:choices}
Let $\tau$ be a simplex. Let $\sigma_1$ a maximal simplex associated to $\tau$ and let $\sigma_2$ be any other maximal simplex such that 
$\tau=\sigma_1\cap\sigma_2$.  Then there exists an element $\phi\in\Mod(S)$, fixing 
$\tau$, such that  for each 
$x$ in the cone over $\tau$ there is a point $X\in\Teich(S)$ with $\pi(X)=\widetilde{\Psi}(x)$ and such that the $X$-length of any curve in $(\sigma_1-\tau)\cup (\phi(\sigma_2)-\tau)$ 
is bounded above by a universal constant, and below by the fixed $\epsilon$.
\end{lemma}

\proof
The coordinates
for curves in $\sigma_1-\tau$ on the cone over $\tau$ are $0$.
By definition, each curve $\beta\in\sigma_1-\tau$ then has fixed
length $\epsilon$ on some $X$ with $\pi(X)=\widetilde\Psi(x)$.
The curves in $\sigma_2-\tau$ may have large intersection with
curves in $\sigma_1-\tau$ and therefore large length on $X$.
However, since there are only finitely many combinatorial types of
pants decompositions, we can choose $\phi$ fixing $\tau$ so that
any curve in $\phi(\sigma_2)-\tau$ has universally bounded
intersection with any curve in $\sigma_1-\tau$.  Since
$\ell_X(\beta)=\epsilon$ for each $\beta\in\sigma_1-\tau$, the
collar about $\beta$ has diameter bounded above.  Thus we can
further compose $\phi$ by Dehn twists about $\beta$, so that for
the new $\phi$, the curves in $\phi(\sigma_2)-\tau$ have bounded
lengths on $X$.
\endproof

\subsection{Properties of $\Psi$}

Our goal in this subsection is to prove that $\Psi$ is a $(1,D)$-quasi-isometry.  In order to do this we will need the following setup.

Let $\sigma$ a maximal simplex. Recall $P$ is the quotient map from  $\CC(S)$ to $\CC(S)/\Mod(S)$. 
Let $d_{P(\sigma)}$ be the path metric on the cone over  $P(\sigma)$ and  let 
 $d^{\M(S)}_{P(\sigma)}$ be the path metric on the (connected) $\Psi$ image of the cone over $P(\sigma)$  in $\M(S)$  induced from the Teichm\"{u}ller metric on $\M(S)$.  
That is, the distance between two points in the image is the infimum of the lengths of  paths joining the points that stays in the image of the cone over $P(\sigma)$.
\begin{lemma}
\label{lem:compare}
There is a constant $D_0$ such that if $x_1,x_2$ lie in the cone over $P(\sigma)$, then 
$$|d_{P(\sigma)}(x_1,x_2)-d_{P(\sigma)}^{\M(S)}(\Psi(x_1),\Psi(x_2))|\leq D_0.$$

\end{lemma}

\begin{proof}
We may find a lift $X_i$ of $\Psi(x_i)$ to $\Teich(S)$ such that 
the difference of the twist coordinates of $X_1$ and $X_2$ with respect to the Fenchel-Nielsen coordinates defined by $\sigma$ are bounded and such that 
$$d_{\Teich(S)}(X_1,X_2)=d_{P(\sigma)}^{\M(S)}(\Psi(x_1),\Psi(x_2)).$$
If $x_1$ and $x_2$ lie in the open cone over $P(\sigma)$, then the lemma follows from Theorem~\ref{thm:Min} 
and the definition of the metric $d_{P(\sigma)}$.  If not, then one must further quote Lemma \ref{lemma:choices} and Lemma \ref{lemma:ratios}.
\end{proof}

\bigskip
\noindent
{\bf $\Psi$ is almost onto: }By a theorem of Bers, there is a
constant $C=C(g,n)$ such that each $X\in\Teich(S)$ has a pants
decomposition corresponding to a maximal simplex $\sigma$ such
that each curve of $\sigma$ has length at most $C$ on $X$.  With
respect to these pants curves, each of the twist coordinates is
bounded, modulo the action of Dehn twists about the curves in
$\sigma$, by $2\pi C$.  Let $\tau$ be the possibly empty face of $\sigma$ such that
the set of curves in $\sigma-\tau$ have lengths on $X$ between
$\epsilon$ and $C$.  The curves in $\tau$ have length at most $\epsilon$. 
By Lemma \ref{lemma:choices}, there is a
point $Y\in\Teich(S)$ such that $\pi(Y)$ is in the $\Psi$-image of the cone
on $\tau$, and such that the lengths on $Y$ of the curves in $\tau$
are the same as the lengths on $X$ of those curves, and the
curves in $\sigma-\tau$ have bounded length on $Y$. Thus their
ratios to the lengths on $X$ are bounded.  Applying
Lemma~\ref{lemma:ratios}, we are done.

\bigskip
\noindent
{\bf $\Psi$ is an almost isometry: }  We need the following lemma.

\begin{lemma}[Path Lemma]
\label{lemma:path}
The following statements are true.

\begin{enumerate}
\item Any two points in $\V(S)$   can be joined by a geodesic
that enters the cone over each  $P(\sigma)$, where $\sigma$ is 
a  maximal simplex of $\widetilde{V(S)}$,  at most once.
\item There is a constant $C'$ such that  any two points of $\Psi(\V(S))$ can be joined by a $(1,C')$ quasi-geodesic in the metric $d_{\M(S)}$ that enters the cone over each $P(\sigma)$  
at most once.
\end{enumerate}
\end{lemma}

A first step in proving Lemma \ref{lemma:path} is the following.

\begin{lemma}
\label{lemma:path2}
The following statements are true.

\begin{enumerate}
\item Suppose $x,y$ are points in the cone over $P(\sigma)$ where $\sigma$ is a maximal simplex. 
Then there is a geodesic joining $x$ and $y$ that stays in the cone over that $P(\sigma)$. 
\item There is a constant $C''$ such that if  $\Psi(x),\Psi(y)$ lie in the cone over $P(\sigma)$ then 
$$d_{P(\sigma)}^{\M(S)}(\Psi(x),\Psi(y))\leq d_{\M(S)}(\Psi(x),\Psi(y))+C''.$$
\end{enumerate}
\end{lemma}

We note that the opposite 
inequality $$d_{\M(S)}(\Psi(x),\Psi(y))\leq d_{P(\sigma)}^{\M(S)}(\Psi(x),\Psi(y))$$ is clearly true. 

\begin{proof}[of Lemma \ref{lemma:path2}]
We prove the first statement.  Lift to $\widetilde{\V}(S)$ and consider again $x,y$ with the same names such that the distance 
in the cone over $\sigma$ realizes the distance between $x$ and $y$ in the cone over $P(\sigma)$. Let the coordinates of $x,y$ be given by  
$(x_1,\ldots, x_d), (y_1,\ldots, y_d)$.  Suppose $\sigma$ is defined by the curves 
$\alpha_1,\ldots,\alpha_d$ of a pants decomposition.  Without loss of generality assume that  
$d_\sigma(x,y)=\frac{1}{2}(y_1-x_1)$. 
We must show that, for every $\phi\in\Mod(S)$, that does not fix $\sigma$, there is no shorter path 
$\rho$ in $\widetilde{\V}(S)$ from $\phi(x)$ to
$y$.

Suppose first that $\alpha_1$ is not a vertex in the simplex $\phi(\sigma)$. 
Then the path from $x$ to $y$ for a last time must enter the cone over 
a simplex for which $\alpha_1$ is a vertex at a point $z$. At $z$ the coordinate corresponding to $\alpha_1$ is $0$, and so $$d_{\V(S)}(y,z)\geq 
y_1/2\geq d_\sigma(x,y).$$ 
Thus we may 
assume that the path $\rho$ joining $\phi(x)$ to $y$ lies completely 
in the cones over simplices for which $\alpha_1$ is a vertex.  Break up this
path into segments 
$\rho=\rho_1*\rho_2*\ldots *\rho_N$, where each $\rho_i$ lies in the cone over a single
simplex. Let $z_1^i$ (resp. $z_1^{i+1}$) be the coordinate of  $\alpha_1$ at the
beginning  (resp. end) of $\rho_i$, where $z_1^1=x_1$ and $z_1^{N+1}=y_1$. 
Then $|\rho_i|\geq \frac{1}{2}|z_1^{i+1}-z_1^i|$.
Thus $$|\rho|=\sum_{i=1}^N |\rho_i|\geq \sum\frac{1}{2} |z_1^{i+1}-z_1^i|\geq \frac{1}{2}(y_1-x_1)=d_{\widetilde{V}(S)}(x,y).$$  We conclude that a shortest path can be found by a geodesic that lies entirely in the cone over $\sigma$

We prove the second statement.  First lift $\Psi(x),\Psi(y)$ to
elements $X,Y\in\Teich(S)$ which lie in
$\Omega_\sigma(\epsilon)$, and such that $$d_{P(\sigma)}^{\M(S)}(\Psi(x),\Psi(y))=d_{\Teich(S)}(X,Y)$$ and whose twist coordinates are
bounded by $2\pi\epsilon$.  By Theorem~\ref{thm:Min}, there
exists a simple closed curve $\alpha_1\in\sigma$ such that
$$|d_{\Teich(S)}(X,Y)-\frac{1}{2}\log(\ell_Y(\alpha_1)/\ell_X(\alpha_1))|\leq
D'$$
where $D'$ depends on $\epsilon$ and on the constant $D$ from
Theorem \ref{thm:Min}.  Thus 
\begin{equation}
\label{eq5}
|d_{P(\sigma)}^{\M(S)}(\Psi(x),\Psi(y))-\frac{1}{2}\log(\ell_Y(\alpha_1)/\ell_X(\alpha_1))|\leq
D'.
\end{equation}

Now let $\phi$ be a mapping class group element.  If $\alpha_1$ is
not a vertex of $\phi(\sigma)$ then any path $\rho$ from
$\phi(Y)$ to $X$ must enter a set 
 $\Omega_{\cal C}(\epsilon)$
for some $\cal C$ containing $\alpha_1$ a last time.  At that
time the length of $\alpha_1$ is $\epsilon$.  By Theorem
\ref{thm:Min} and Equation (\ref{eq5}) we then have $$|\rho|\geq
\frac{1}{2}\log(\epsilon/\ell_X(\alpha_1))-D\geq
d_{P(\sigma)}^{\M(S)}(\Psi(y),\Psi(x))+
\frac{1}{2}\log(\epsilon/\ell_Y(\alpha_1))-D-D'.$$
  Since $\ell_Y(\alpha_1)$ is bounded above, the 
term $\frac{1}{2}\log(\epsilon/\ell_Y(\alpha_1))-D-D'$ is bounded below by some constant, and we 
set $-C''$ to be this constant.

Thus again we can assume that $\rho$ lies completely in $\Omega_{\cal C}(\epsilon)$ for a set $\cal C$ containing $\alpha_1$.  But now the conclusion again follows from Theorem~\ref{thm:Min}. 
\end{proof}

\begin{proof}[of Lemma \ref{lemma:path}]
Suppose $x$ is in the cone over $P(\sigma_1)$ and that $y$ is in the
cone over $P(\sigma_2)$.  If $P(\sigma_1)=P(\sigma_2)$ then we
are done by Lemma \ref{lemma:path2}.  Thus we can assume that
$P(\sigma_2)\neq P(\sigma_1)$.  Suppose $\rho$ is a geodesic from
$x$ to $y$.  Suppose $\rho$ leaves the cone over $P(\sigma_1)$
and returns to it for a last time at some $z$ in the cone over
$P\sigma_1)\cap P(\sigma_3)$ for some maximal simplex
$\sigma_3$. Then by the first part of Lemma \ref{lemma:path2} we
can replace $\rho$ by a geodesic that stays in the cone over
$P(\sigma_1)$ from $x$ to $z$ and then follows $\rho$ from $z$ to
$y$ never returning to the cone over $P(\sigma_1)$. We now find
the last point $w$ that lies in the cone over $P(\sigma_3)$ and
replace a segment of $\rho$ with one that stays in the cone over
$P(\sigma_3)$ and never returns again to the cone over
$P(\sigma_3)$. Since there are only a finite number of simplices
in $\CC(S)/\Mod(S)$, continuing to apply Lemma \ref{lemma:path2},  we are done.  This proves the first
statement.

The proof of the second statement is similar, where we now use the second part of  Lemma \ref{lemma:path2}.
\end{proof}

We now continue with the final step in the proof of Theorem
\ref{theorem:model}, that the map $\Psi$ is an almost isometry.
We first prove that $$d_{\M(S)}(\Psi(x),\Psi(y))\leq
d_{\V(S)}(x,y)+R$$ for some constant $R$.  To prove this,
consider a geodesic path $\gamma\subset\V(S)$ connecting $x$ to
$y$.  By the first statement of Lemma \ref{lemma:path}, there
exists $c=c(S)$ so that $\gamma$ can be written as a
concatenation $\gamma=\gamma_1\ast\cdots\ast\gamma_c$ with each
$\gamma_i$ a geodesic in the cone over $P(\sigma)$ for $\sigma$ a maximal simplex
$\sigma_i$ of $\V(S)$.  By Lemma~\ref{lem:compare} each
$\Psi(\gamma_i)$ is a $(1,D_0)$-quasigeodesic in the metric
$d_{\M(S)}$.  It follows that $\Psi(\gamma)$ is a
$(1,cD')$-quasigeodesic.

The proof of the opposite inequality 
$$d_{\V(S)}(x,y)\leq d_{\M(S)}(\Psi(x),\Psi(y))+R'$$ 
for some $R'$ uses the second conclusion  of Lemma \ref{lemma:path}.  Any two points can be joined by $(1,C')$ quasi-geodesic in the metric $d_{\M(S)}$ and which intersects a fixed number of cones over image simplices $P(\sigma)$.   We now apply Lemma~\ref{lem:compare} to conclude that $d_{\V(S)}(x,y)$ is only larger by an additive constant.

\noindent
Dept. of Mathematics, University of Chicago\\
5734 University Ave.\\
Chicago, Il 60637\\
E-mail: farb@math.uchicago.edu, masur@math.uchicago.edu

\end{document}